\documentclass[secnum,leqno]{rims-bessatsu}
\usepackage{amssymb, amsmath}
\usepackage{mathrsfs}
\usepackage{pb-diagram}
\usepackage[all]{xy}
\usepackage{graphics}
\usepackage[dvips]{graphicx}
\usepackage{eepic}
\usepackage{epic}
\usepackage{enumerate,mathtools}
\newtheorem{thm}{Theorem}[section]
\newtheorem{crl}[thm]{Corollary}

\newtheorem{lmm}[thm]{Lemma}

\newtheorem{prp}[thm]{Proposition}

\theoremstyle{definition}
\newtheorem{dfn}[thm]{Definition}
\newtheorem{nota}[thm]{Notation}
\theoremstyle{remark}
\newtheorem{rem}[thm]{Remark}
\newenvironment{eq-text}
{\begin{equation} \begin{minipage}[t]{0.85\linewidth}}
{\end{minipage} \end{equation} \ignorespacesafterend}

\newcommand{\ti}{\tilde}
\newcommand{\pa}{\partial}
\newcommand{\dd}{{\mathrm d}}
\newcommand{\un}[1]{\underline{#1}}
\newcommand{\ov}[1]{\overline{#1}}

\newcommand{\wt}{\widetilde}

\newcommand{\upa}{\uparrow}

\newcommand{\isom}{\xrightarrow{\smash{\ensuremath{\sim}}}}
\newcommand{\be}{\beta}
\newcommand{\eps}{\varepsilon}
\newcommand{\ph}{\varphi}
\newcommand{\Om}{\Omega}
\newcommand{\om}{\omega}
\newcommand{\Ga}{\Gamma}
\newcommand{\ga}{\gamma}
\newcommand{\De}{\Delta}
\newcommand{\de}{\delta}
\newcommand{\la}{\lambda}

\newcommand{\ze}{\zeta}

\DeclareMathOperator{\Z}{\mathbb{Z}}
\DeclareMathOperator{\R}{\mathbb{R}}
\DeclareMathOperator{\C}{\mathbb{C}}

\newcommand{\sD}{\mathscr D}
\newcommand{\sE}{\mathscr E}

\newcommand{\sL}{\mathscr L}
\newcommand{\sR}{\mathscr R}
\newcommand{\sO}{\mathscr O}
\newcommand{\sT}{\mathscr T}

\newcommand{\cB}{\mathcal{B}}

\newcommand{\cM}{\mathcal{M}}
\newcommand{\cN}{\mathcal{N}}

\newcommand{\cS}{\mathcal{S}}

\newcommand{\dfs}{d.f.s.}
\newcommand{\fp}{\mathfrak{p}}
\newcommand{\fq}{\mathfrak{q}}
\newcommand{\uO}{\un{0}}

\newcommand{\uga}{\un{\ga}}

\newcommand{\Rp}{\R_{\geq0}}
\newcommand{\dist}{\operatorname{dist}}
\newcommand{\Det}{\operatorname{det}}

\newcommand{\dst}{\displaystyle}

\title{
Resurgent functions and
nonlinear systems of
differential and difference equations
}
\TitleHead{
Resurgent functions and
nonlinear systems of
differential and difference equations
}          
\author{Shingo \textsc{Kamimoto}
\footnote{
Graduate School of Sciences,
Hiroshima University, 
1-3-1 Kagamiyama, Higashi-Hiroshima,
Hiroshima 739-8526, Japan
} 
}

\AuthorHead{Shingo Kamimoto }         
\support{Supported by \underline{Grant-in-Aid for JSPS Fellows
Grant Number 15J06019}}  
\begin{document}
%

\maketitle

\thispagestyle{empty}



\begin{abstract}
The principal aim of this article is
to establish an iteration method
on the space of resurgent functions.
We discuss endless continuability of
iterated convolution products of
resurgent functions
and derive their estimates
developing the method in \cite{KS}.
Using the estimates,
we show the resurgence of
formal series solutions of nonlinear
differential and difference equations.
\end{abstract}

\section{Introduction}\label{sec:1}

Resurgent analysis has its origin in the publications
\cite{E} written by J. \'Ecalle.
It provides an effective method
for the study of e.g. holomorphic dynamics,
analytical differential equations,
WKB analysis and
it still fascinates many mathematicians
and theoretical physicists.
In this theory,
the space of resurgent functions plays a central role:
a formal series 
$
\ph(x):=\sum_{j=0}^\infty
\ph_jx^{-j}
\in\C[[x^{-1}]]
$
is resurgent if its formal Borel transform
\begin{equation*}
\cB(\ph):=
\ph_0\de + \hat\ph(\xi),
\qquad
\hat\ph(\xi):= \sum_{j=1}^{\infty} \ph_j \frac{\xi^{j-1}}{(j-1)!}
\end{equation*}
is convergent and
$\hat\ph(\xi)$ is endlessly continuable (cf. \cite{E}).
In this article,
we adopt the definition of endless continuability
in \cite{CNP}.
As the Borel counterpart of Cauchy product
in $\C[[x^{-1}]]$,
the convolution product $\hat{\ph}*\hat{\psi}$
of $\hat{\ph}$ and $\hat{\psi}$ in $\C\{\xi\}$
is defined as follows:
\begin{equation*}
\hat{\ph}*\hat{\psi}(\xi)
:=\int_0^\xi
\hat{\ph}(\xi-\xi')\hat{\psi}(\xi')d\xi'.
\end{equation*}
%
To discuss analytic continuation
of such a convolution product,
the notion of
\emph{symmetrically contractible path}
was introduced in \cite{E}.
Following the principle in \cite{CNP},
systematic construction of such paths
was given in \cite{S2} and
detailed estimates for the convolution product
of an arbitrary number 
of endlessly continuable functions
were obtained in \cite{S3}
when the set of singular points of
the functions is a closed discrete subset in $\C$
and closed under addition (see also \cite{MS}).
Further,
it was
generalized in \cite{OD}
and \cite{KS} to the case
where the location of singular points
of endlessly continuable functions
is written by a discrete filtered set
(see Definition \ref{dfn:2.1} for its definition).
Especially in \cite{KS},
a rigorous foundation
for the analysis on the space of
such endlessly continuable functions
was provided:
a structure of Fr\'echet space
on the space of the functions was
precisely given by the aid of endless Riemann
surfaces (see Section \ref{sec:2}).
It allows us to handle analytical problems
related to the convergence of the functions,
e.g. substitution of resurgent functions
to convergent series,
implicit function theorem for resurgent functions.

However,
we have still a problem in applying
resurgent analysis to the study of analytical differential equations:
there is no universally applicable way of proving
the resurgence of formal series solutions of
differential equations.
Especially,
it is important to determine
the location of singular points
of the Borel transformed formal series solutions
for the use of alien calculus, which is the main tool of
resurgent analysis.

Having these backgrounds in mind,
we discuss the following question in this article:
Can we extend the principle in \cite{KS}
so that we can show the resurgence of
formal series solutions of
differential equations?
The main purpose of this article
is to establish an iteration method
on the space of resurgent functions by
developing the method in
\cite{KS}
and to show the resurgence of formal series solutions
of differential equations by applying it.
More precisely, we consider 
a nonlinear differential equation
\begin{equation}\label{1.1}
\frac{d}{dx}\Phi
=F(x^{-1},\Phi)
\end{equation}
at $x=\infty$ with
$F(x^{-1},\Phi)\in\C^n\{ x^{-1},\Phi\}$
satisfying the conditions
$
F(0,0)=0
$
and
$
{\rm det} \big(\partial_{\Phi}F(0,0)\big)
\neq0.
$
In this setting,
\eqref{1.1} has a unique formal series solution
$\Phi\in\C^n[[x^{-1}]]$.
In \cite{E}, \'Ecalle claims
each entry of $\Phi$ is resurgent.
In \cite{Co},
Borel summability of 
transseries solutions of \eqref{1.1}
was discussed
and the singularity structure in the Borel plane
of the solutions was precisely studied
under non-resonance conditions.
(See \cite{Ku}, \cite{BKu} and \cite{BKu2}
for the case of the difference equation \eqref{7.13}.)
However, our standpoint is close to \'Ecalle's
\emph{mould calculus}
rather than \cite{Co}.
Mould calculus was developed in \cite{E}
and applied to the classification
of saddle-node singularities in \cite{E2}
(see also \cite{S8} and \cite{S9}).
It uses expansions by resurgent monomials
associated with words generated by e.g. $\Z$
and resurgent properties of formal integrals
was studied by the use of the mould expansions.
In this article,
we use an expansion of $\hat\Phi$ by 
iterated convolution products 
of meromorphic functions
associated with \emph{iteration diagrams} 
(see Definition \ref{dfn:3.1})
instead of words.
Iterated convolution product
is a combination of convolution
product and Cauchy product
determined by iteration diagrams
(see Definition \ref{dfn:3.6}).
Extending the estimates obtained in \cite{KS}
to iterated convolution products 
of endlessly continuable functions,
we obtain the following theorem as
one of our main results:
\begin{thm}
The formal series solution $\Phi\in\C^n[[x^{-1}]]$
of \eqref{1.1} is resurgent.
\end{thm}
In Section \ref{sec:7},
we describe detailed geometrical structure of
singular points of $\hat\Phi$
by the use of discrete filtered set
and reveal how the singular points
are generated by the set of eigenvalues of 
$\partial_{\Phi}F(0,0)$.

The plan of this article is the following:
\begin{enumerate}[--]
\item
Section \ref{sec:2} reviews 
the notions and the results
related to $\Om$-resurgence.

\item
Section \ref{sec:3}
introduces the notions of
iteration diagram and iterated convolution.
We give a key-estimate Theorem \ref{thm:3.9} for 
iterated convolution products of
$\Om$-resurgent functions.

\item
Section \ref{sec:4} discusses the analytic continuation
of iterated convolution products along a path $\ga$
using a $(\ga,T)$-adapted deformation.

\item
Section \ref{sec:5} and Section \ref{sec:6}
are devoted
to the proof of Theorem \ref{thm:4.6}:
We construct a $(\ga,T)$-adapted deformation
$(\Psi_t)_{t\in[a,1]}$ in Section \ref{sec:5}
and derive its estimates in Section \ref{sec:6}.

\item
In Section \ref{sec:7},
we show the resurgence of formal series solutions
of nonlinear differential and difference equations
using the estimate Theorem \ref{thm:3.9}.

\end{enumerate}

Some of the results in this article
have been announced in \cite{K}.


\section{Preliminaries}
\label{sec:2}


In this section,
we review the notions concerning
$\Om$-resurgence of formal series
discussed in \cite{KS}.

\begin{dfn}\label{dfn:2.1}
We use the notation
$\Rp = \{\la\in\R\mid\la\geq0\}$.
\begin{enumerate}
\item
A \emph{discrete filtered set}, or \emph{\dfs}\ for short, 
is a family $\Om = (\Om_L)_{L\in\Rp}$ of
subsets of~$\C$ such that
\begin{enumerate}[a)]
\item
$\Om_L$ is a finite set, 
\item
$\Om_{L_1}\subseteq \Om_{L_2}$ for $L_1\leq L_2$,
\item
there exists $\de>0$ such that $\Om_\de=\O$.
\end{enumerate}
\item
Let~$\Om$ and~$\Om'$ be \dfs\ 
A \dfs\ $\Om*\Om'$ 
defined by the formula
\[
(\Om *\Om')_L := \{\, \om_1+\om_2 \mid
\om_1\in\Om_{L_1}, \om_2\in\Om'_{L_2}, L_1+L_2=L \, \}
\cup\Om_{L}\cup\Om'_{L}
\ \text{ for $L\in\Rp$}
\]
is called the \emph{sum} of \dfs\ $\Om$ and~$\Om'$.
We set
$\Om^{*n} :=
\underbrace{ \Om* \cdots*\Om }_{\text{$n$ times}}$
for $n\ge1$ and define a \dfs\ 
$\Om^{*\infty}$ by
\[
\Om^{*\infty} :=\varinjlim_n\ \Om^{*n}.
\]
\item
A \emph{trivial} \dfs\ $\Om=(\Om_L)_{L\in\Rp}$
is a \dfs\ satisfying $\Om_L=\O$ for all $L\in\Rp$
and we denote it by $\O$.
\item
Given a \dfs\ $\Om$,
the \emph{distance} to $\Om$
is the number
$\rho(\Om):=\sup\{\,\rho\in\Rp\mid \Om_\rho=\O\,\}$.

\end{enumerate}
\end{dfn}

We define for a \dfs~$\Om$
\[
\cS_\Om:= \big\{ (\la,\om)\in\Rp\times\C \mid \om\in\Om_\la \big\},
\]
\[
\cM_\Om := \big(\Rp\times\C\big) \setminus \ov\cS_\Om,
\]
where $\ov\cS_\Om$ denotes the closure of $\cS_\Om$ in 
$\Rp\times\C$. 

Let $\Pi$ be the set of all Lipschitz paths
$\ga : [0,1] \to \C$ such that $\ga(0)=0$.
We denote
the restriction of~$\ga\in\Pi$ to the
interval $[0,t]$ for $t\in[0,1]$
by $\ga_{|t}$
and the total length of~$\ga_{|t}$
by $L(\ga_{|t})$.

\begin{dfn}   \label{dfn:2.2}
Given a \dfs~$\Om$,
we call $\ga\in\Pi$
\emph{$\Om$-allowed path}
if it satisfies
\[
\ti\ga(t) :=
\big( L(\ga_{|t}), \ga(t) \big) \in \cM_\Om
\quad
\text{for all $t\in[0,1]$,}
\]
and denote the set of all $\Om$-allowed paths
by $\Pi_\Om$.
\end{dfn}

\begin{rem}
When a piecewise $C^1$ path
$t \in [0,1] \mapsto 
\ti\ga(t) = \big( \la(t),\ga(t) \big) \in \cM_\Om$
with $\ti\ga(0)=(0,0)$
is given,
the $\Om$-allowedness of $\ga$ is characterized
by the condition
$\la'(t) = |\ga'(t)|$ for a.e.~$t\in[0,1]$.
\end{rem}
Recall that
an $\Om$-endless Riemann surface is a triple $(X,\fp,\un0)$ 
such that
$X$ is a connected Riemann surface,
$\fp : X \to \C$ is a local biholomorphism,
$\un0 \in \fp(0)$,
and any path $\ga : [0,1] \to \C$ of~$\Pi_\Om$ has a lift
$\uga : [0,1] \to X$ such that $\un\ga(0) = \un0$.
A morphism 
$
\fq:(X,\fp,\un0)\to(X',\fp',\un0')
$
of $\Om$-endless Riemann surfaces
is given by a local biholomorphism $\fq:X\to X'$
such that the following diagram is commutative:
\[
\begin{xy}
(0,20) *{(X,\un0)},
(30,20) *{\ (X',\un0')},
(15,0) *{(\C,0)},
{(6,20) \ar (24,20)},
{(2,16) \ar (13,4)},
{(28,16) \ar (17,4)},
(14,23) *{\fq},
(3,10) *{\fp},
(27,10) *{\fp'}
\end{xy}
\]

The existence of the initial object
$(X_{\Om},\fp_\Om,\un0_\Om)$ in the category
of $\Om$-endless Riemann surfaces
was proved in \cite{KS}:
\begin{thm}[\cite{KS}]   \label{thm:2.4}
There exists an $\Om$-endless Riemann surface 
$(X_{\Om},\fp_\Om,\un0_\Om)$ such
that $X_\Om$ is simply connected
and, for any $\Om$-endless Riemann surface $(X,\fp,\un0)$,
there is a unique morphism 
$$
\fq : (X_{\Om},\fp_\Om,\un0_\Om) \to  (X,\fp,\un0).
$$
\end{thm}

Let $\hat\sR_\Om$ denote the space of
$\Om$-continuable functions,
i.e.,
holomorphic germs $\hat\ph \in \C\{\xi\}$ which can be analytically
continued along any path $\ga\in\Pi_{\Om}$.
Then,
there exists an isomorphism
\[
\fp_\Om^* :
\hat\sR_\Om \isom \Ga(X_{\Om},\sO_{X_{\Om}}),
\]
where $\sO_{X_\Om}$ is the sheaf of holomorphic functions
on $X_\Om$,
and hence,
a structure of Fr\'echet space is
naturally introduced
to $\hat\sR_\Om$ as follows:
We set for $L,\de>0$
$$
\cM_\Om^{\de,L} := \big\{\,
(\la,\xi) \in \Rp\times\C \mid
\dist\big( (\la,\xi),  \ov\cS_\Om \big) \ge \de,
\la\leq L
\big\},
$$
$$
\Pi_\Om^{\de,L} := \big\{\, \ga\in\Pi_\Om \mid
\big(L(\ga_{|t}),\ga(t)\big)\in\cM_\Om^{\de,L}
\;\, \text{for all $t\in[0,1]$} \,\big\},
$$
where 
$\dist(\cdot,\cdot)$ is the Euclidean distance in 
$\R\times\C\simeq \R^3$,
and define compact subsets
$K_{\Om}^{\de,L}$ of $X_\Om$ by
$$
K_{\Om}^{\de,L} := \big\{\, \un\ga\,(1)
\in X_\Om \mid
\ga\in\Pi_\Om^{\de,L}
\,\big\}.
$$
Since
$X_\Om$ is exhausted by $(K_\Om^{\de,L})_{\de,L>0}$,
a family of
seminorms $\|\cdot\|_{\Om}^{\de,L}$ $(\de,L>0)$
defined by
$$
\|\, \hat\ph\,
\|_{\Om}^{\de,L}
:=\sup_{\un\xi\in K_\Om^{\de,L}}
|\, \fp_\Om^{*}\hat\ph(\un\xi)|
\quad
\text{for}
\quad
\hat\ph\in\hat\sR_\Om
$$
induces
a structure of
Fr\'echet space on 
$\hat\sR_\Om$.
Correspondingly,
a family of seminorms
$\|\cdot\|_{\Om}^{\de,L}$ $(\de,L>0)$
on the space of $\Om$-resurgent series
\[
\sR_\Om := \cB^{-1} \big( \C\de \oplus \hat\sR_\Om \big)
\]
are defined by
$$
\|\, \ph\,
\|_{\Om}^{\de,L}
:=
|\ph_0|+
\|\,\hat\ph\,
\|_{\Om}^{\de,L}
\quad
\text{for}
\quad
\ph\in\sR_\Om,
$$
where 
$
\cB(\ph)
=\ph_0\de+\hat\ph
\in\C\de \oplus \hat\sR_\Om.
$

\begin{rem}
Notice that
$
(X_{\O},\fp_{\O},\un0{}_{\O})
\isom(\C,{\rm id}_{\C},0),
$
and hence,
$\hat\sR_{\O}\isom\Ga(\C;\sO_{\C})$.
Since 
$
K_{\O}^{\de,L}=\{\xi\in\C\mid |\xi|\leq L\}
$
for $\de,L>0$,
we have
$\dst
\|\, \hat\ph\,
\|_{\O}^{\de,L}
=\sup_{|\xi|\leq L}|\hat\ph(\xi)|
$
for $\hat\ph\in\hat\sR_{\O}$.
\end{rem}

Now, let $\Om'$
be a \dfs\ satisfying
$\Om\subset\Om'$.
From Theorem \ref{thm:2.4},
we find that there exists a morphism
$
\fq:(X_{\Om'},\fp_{\Om'},\uO_{\Om'})
\to (X_{\Om},\fp_\Om,\uO_\Om).
$
Since 
$
\fq(K_{\Om'}^{\de,L})\subset K_\Om^{\de,L},
$
we have
$$
\|\,\hat\ph\,
\|_{\Om'}^{\de,L}
\leq
\|\, \hat\ph\,
\|_{\Om}^{\de,L}
\quad
\text{for}
\quad
\hat\ph\in\hat\sR_\Om,
$$
and hence,
$$
\|\, \ph\,
\|_{\Om'}^{\de,L}
\leq
\|\, \ph\,
\|_{\Om}^{\de,L}
\quad
\text{for}
\quad
\ph\in\sR_{\Om}.
$$

\section{Iterated convolution of resurgent functions}
\label{sec:3}
%
\subsection{Iteration diagram}
%
\begin{dfn}\label{dfn:3.1}
Let $T=(V,E)$ be a directed tree diagram,
where $V$ (resp. $E$) is the set of vertices (resp. edges)
of $T$.
We call $T$ \emph{iteration diagram} if $T$ satisfies
the condition that
any vertex $v\in V$ has 
at most one outgoing edge.
We denote the set of iteration diagrams by
$\mathscr{T}$.
\end{dfn}
Since $T\in\mathscr{T}$ is connected and has no cycles,
we immediately have the following
\begin{lmm}\label{lmm:3.2}
Each $T\in\mathscr{T}$ has a unique vertex $\hat v$ such that
there exists a path 
$v\to \cdots \to \hat v$
from $v$ to $\hat v$ in $T$
for any vertex $v\in V$
and such a path is unique.
\end{lmm}
\begin{dfn}
Let $T=(V,E)$ be an iteration diagram.
\begin{enumerate}
\item
We call $\hat v$ in Lemma \ref{lmm:3.2}
\emph{root} of T.

\item
We call a vertex $v$ \emph{leaf}
of $T$ if $v$ has no edge $e$
such that the terminal vertex of $e$ is $v$
and denote the set of leaves of $T$ by $L$.

\item
The \emph{branch} $T_v=(V_v,E_v)$ of $T$ at $v\in V$
is the diagram that consists of the vertexes
$u\in V$
that have a path 
$u\to \cdots \to v$
from $u$ to $v$ in $T$
and the edges 
$v_1\overset{e}{\to}v_2\in E$ such that
$v_1,v_2\in V_v$.
\end{enumerate}
\end{dfn}
From the definition of the branch,
we obtain the following
\begin{lmm}
Given $T\in\sT$,
the branch $T_v$ of $T$
at each vertex $v\in V$
defines an iteration diagram
with the root $v$.
\end{lmm}
\begin{nota}
Let $T=(V,E)$ be an iteration diagram.
\begin{enumerate}
\item
We set $V^{\circ}:=V\setminus\{\hat v\}.$

\item
For each $v\in V^{\circ}$,
there exists a unique vertex $u$ that has
an edge $v\to u$.
We denote such a vertex by $v_\upa$.

\item
Given $v\in V$,
we denote
the set of vertices $u\in V$
that have an edge $u\to v$ by
$
V_v^1.
$
\end{enumerate}
\end{nota}

We assign each vertex $v$ a weight $w_v$
defined as the cardinal of
$
\{
v'\in L\mid
\exists\,\text{a path}\ 
v'\to\cdots\to v
\}.
$
Notice that $w_v$ satisfies
%
%
$$
\left\{
\begin{aligned}
w_v&=1&\qquad&(v\in L),\\
w_v&=\sum_{u\in V_v^1}w_u
&\qquad&(v\in V\setminus L).
\end{aligned}
\right.
$$
%
Iteration diagrams are graded by the cardinal $|V|$
of vertexes:
$$
\sT=
\bigsqcup_{k=1}^\infty
\sT_k,
\quad
\sT_k=\{T=(V,E)\in\sT\mid
|V|=k\}.
$$

\subsection{Iterated convolution}

Let $T=(V,E)\in\sT_k$ $(k\geq1)$ be an iteration diagram
and assume that
analytic germs
$\hat f_v,\hat \ph_v\in\C\{\xi\}$ are
assigned to each vertex $v\in V$.
Starting from the leaves of $T$,
we inductively construct
$\{\hat{\psi}_v\}_{v\in V}$
from 
$\{\hat{f}_v\}_{v\in V}$
and
$\{\hat{\varphi}_v\}_{v\in V}$
by the rule
\begin{equation}\label{3.1}
\hat{\psi}_v
:=
\hat{\varphi}_v\cdot
\Big(\hat{f}_v*
\sideset{}{^*}\prod_{u\in V_v^1}
\hat{\psi}_u
\Big)
\quad
(v\in V),
\end{equation}
where
$\dst
\sideset{}{^*}\prod_{u\in V_v^1}\hat{\psi}_u
$
is the convolution product of $\hat{\psi}_u$
over all the vertices
$u\in V_v^1$
and we regard it as the unit $\delta$
when $v\in L$.

%
%
\begin{dfn}\label{dfn:3.6}
Given $T\in\sT$ and 
$\{\hat{f}_v\}_{v\in V},
\{\hat{\varphi}_v\}_{v\in V}\subset\C\{\xi\}$,
we call $\hat \psi_{T}:=\hat \psi_{\hat v}$
defined by the rule \eqref{3.1}
\emph{iterated convolution} of 
$\big(\,T\,;\{\hat{f}_v\}_{v\in V},\{\hat{\varphi}_v\}_{v\in V}\big)$.
\end{dfn}

\begin{nota}
For an iteration diagram
$T=(V,E)\in\sT_k$ $(k\geq1)$,
we set
$$
\Delta_T
:=
\Big\{(s_v)_{v\in V}
\in\Rp^k
\,\Big|\,
\sum_{u\in V_v^1}s_u\leq s_v,
s_{\hat{v}}=1
\Big\}
$$
and 
$[\De_T]\in\sE_{k-1}(\R^k)$ denotes the corresponding integration
current,
where the orientation of $\Delta_T$ is defined
so that it satisfies
\begin{equation}
\int_{\Delta_T}\bigwedge_{v\in V^\circ}\dd s_v
= \frac{1}{(k-1)!}.
\end{equation}
\end{nota}

We set for $\rho>0$
$$ 
D_{\rho} := \{\, \xi\in\C\mid |\xi| < \rho \,\}.
$$
Let $\Om$ be a \dfs\ 
We define a map $\sL_v:D_{\rho(\Om)}\to X_{\Om_v}$ by
$$
\sL_v(\xi):=\un{\ga}{}_\xi(1)
\quad\text{for}\quad \xi\in D_{\rho(\Om)},
$$
where $\Om_v:=\Om^{*w_v}$ $(v\in V)$
and
$\un\ga{}_\xi:[0,1]\to X_{\Om_v}$
is the lift of the path
$\ga_\xi:t\in[0,1]\mapsto t\xi$.
Notice that $\sL_v$ gives a local isomorphism
from $D_{\rho(\Om)}$ to an open neighborhood $\sL_v(D_{\rho(\Om)})$
of $\un0{}_{\Om_v}\in X_{\Om_v}$.
%
%

Now, assume that
$\{\hat{f}_v\}_{v\in V}\subset\hat\sR_{\O}$
and
$\hat{\varphi}_v\in\hat\sR_{\Om_v}$
for $v\in V$.
We consider a map 
\[
\sD(\xi) :
\vec{s}=(s_v)_{v\in V} \mapsto 
\sD(\xi,\vec s\,) := \big( \sL_v(s_v\xi) \big)_{v\in V} \in
X_\Om^T
\quad\text{for}\quad
\xi\in D_{\rho(\Om)}
\]
defined on a neighborhood of $\Delta_T$ in $\R^k$,
where
$$
X_\Om^T :=
\prod_{v\in V} X_{\Om_{v}}.
$$
Let
$\sD(\xi)_\# [\De_T] \in \sE_{k-1}(X_\Om^T)$
denote
the push-forward
of~$[\De_T]$ by~$\sD(\xi)$.
Then, 
we have the following representation of 
the iterated convolution $\hat{\psi}_{T}$ of 
$\big(\,T\,;\{\hat{f}_v\}_{v\in V},\{\hat{\varphi}_v\}_{v\in V}\big)$:

\begin{prp}\label{prp:3.7}
Given $\big(\,T\,;\{\hat{f}_v\}_{v\in V},\{\hat{\varphi}_v\}_{v\in V}\big)$,
define a holomorphic $(k-1)$-form $\beta_T$
on $X_\Om^T$ by
\[ 
\beta_T
:=\Big(\prod_{v\in V}
(\fp_{\Om_v}^*\hat{\varphi}_v)(\un\xi{}_v)
\hat{f}_v\big(\xi_v-\sum_{u\in V_v^1}\xi_u\big)
\Big)
\bigwedge_{v\in V^\circ}\dd \un\xi{}_v,
\]
where
$\dst\bigwedge_{v\in V^\circ}\dd \un\xi{}_v$ is
the pullback
of the $(k-1)$-form $\dst\bigwedge_{v\in V^\circ}\dd \xi_v$
in $X_\Om^T$ by 
$(\fp_{\Om_{v}})_{v\in V}: X_\Om^T \to \C^k$
and $\xi_v=\fp_{\Om_v}(\un\xi{}_v)$ $(v\in V)$.
Then, the following equality
holds for $\xi\in D_{\rho(\Om)}$:
\begin{equation}\label{3.4}
\hat{\psi}_{T}(\xi)=
\sD(\xi)_\# [\De_T](\be_T).
\end{equation}
\end{prp}

\begin{proof}
%
%
We prove \eqref{3.4}
by induction.
We first note that
$\sD(\xi)_\# [\De_{T}](\be_{T})$
is regarded as
$
(\fp_\Om^*\hat{\varphi}_{\hat v})(\un\xi{}_{\hat v})
\hat{f}_{\hat v}(\xi_{\hat v})\big|_{\un\xi{}_{\hat v}=\sL(\xi)}
=
\hat{\varphi}_{\hat v}(\xi)
\hat{f}_{\hat v}(\xi)
$
when $T\in\sT_1$,
and hence, the equality
\eqref{3.4} holds for the diagram $T_v$ $(v\in L)$.
Next, take $v\in V$ and
assume that \eqref{3.4} holds for all the branches
$T_u$ $\big(u\in V_v^\circ\big)$.
From the definition of the iterated convolution
\eqref{3.1},
we have the following representation of $\hat\psi_{T_v}$:
$$
\hat\psi_{T_v}(\xi)
=
\hat{\varphi}_v(\xi)
\int_{\De_{\ell}}
\hat{f}_v\big(\xi \big(1-\sum_{u\in V_v^1}s_u\big)\big)
\prod_{u\in V_v^1}
\hat{\psi}_{T_u}(\xi s_u)
\bigwedge_{ u\in V_v^1}\xi\dd s_{u},
$$
where 
$\ell=|V_v^1|$ and $\De_\ell$ is
the $\ell$-dimensional simplex defined by
$$
\De_\ell
:=
\big\{(s_u)_{u\in V_v^1}\in\Rp^{\ell}
\,\big|\,
\sum_{u\in V_v^1} s_u\leq1
\big\}.
$$
Since
$
\Delta_{T_v}
$
is rewritten as
$$
\left\{(s_{\ti v})_{\ti v\in V_v}
\ \left|\ 
\begin{aligned}
&
s_{\ti v}=s_{ u}\ti s_{\ti v},
(\ti s_{\ti v})\in \De_{T_u},
(s_{u})\in \De_\ell
\\
&
\ \text{for}\ 
\ti v\in V_u
\ \text{and}\ 
u\in V_v^1,
s_v=1
\end{aligned}
\right.
\right\},
$$
we obtain \eqref{3.4} for the diagram $T_v$
from the induction hypothesis.
It proves \eqref{3.4} for $T=T_{\hat v}$.
\end{proof}

We now state one of our main theorems:
\begin{thm}\label{thm:3.9}
Let $\Om$ be a \dfs\ and let $\de,L>0$ be reals 
such that $2\de<\rho(\Om)$.
Then, there exist $c,\de'>0$ such that,
for every
$T=(V,E)\in\mathscr{T}_k$ $(k\geq1)$,
$\{\hat{f}_v\}_{v\in V}\subset\hat\sR_{\O}$
and
$\hat{\varphi}_v\in\hat\sR_{\Om_v}$ $(v\in V)$,
the iterated convolution $\hat{\psi}_{T}$ of 
$\big(\,T\,;\{\hat{f}_v\}_{v\in V},\{\hat{\varphi}_v\}_{v\in V}\big)$
is $\Om_{\hat{v}}$-continuable
and satisfies the following estimates:
\begin{equation}\label{3.5}
\big\|
\hat\psi_T
\big\|_{\Om_{\hat{v}}}^{\de,L}
\leq
\frac{c ^{k-1}}{(k-1)!}
\sup_{\vec{s}
\in\Delta_T
}
\prod_{v\in V}
\big\|\hat{\varphi}_v\big\|_{\Om_v}^{\de',s_vL}
\big\|\hat{f}_v\big\|_{\O}^{\de',s_vL}.
\end{equation}
\end{thm}

Since $\Om\subset\Om_v$ for all $v\in V$,
we obtain the following
\begin{crl}\label{crl:3.10}
Under the same assumptions
with Theorem \ref{thm:3.9},
there exist $c,\de'>0$ such that,
for every
$T=(V,E)\in\mathscr{T}_k$ $(k\geq1)$,
$\{\hat{f}_v\}_{v\in V}\subset\hat\sR_{\O}$
and
$\{\hat{\varphi}_v\}_{v\in V}\subset\hat\sR_{\Om}$,
the iterated convolution $\hat{\psi}_{T}$ of 
$\big(\,T\,;\{\hat{f}_v\}_{v\in V},\{\hat{\varphi}_v\}_{v\in V}\big)$
satisfies 
\begin{equation}\label{3.6}
\big\|
\hat\psi_T
\big\|_{\Om_{\hat{v}}}^{\de,L}
\leq
\frac{c ^{k-1}}{(k-1)!}
\prod_{v\in V}
\big\|\hat{\varphi}_v\big\|_{\Om}^{\de',L}
\big\|\hat{f}_v\big\|_{\O}^{\de',L}.
\end{equation}
\end{crl}

The proof of Theorem \ref{thm:3.9}
will be given in Section \ref{sec:4}.

\section{$(\ga,T)$-adapted deformation of 
$\sD\big( \ga(a) \big)$
}\label{sec:4}

In this section,
we introduce the notion of
$(\ga,T)$-adapted deformation of 
$\sD\big( \ga(a) \big)$,
which is a slight generalization of $\ga$-adapted
origin-fixing isotopies in \cite[Def.~5.1]{S3}.
Let $T=(V,E)$ be an iteration diagram
and let $\Om$ be a \dfs\ 
We take $\rho>0$ such that $2\rho<\rho(\Om)$.
We fix a path
$\ga : [0,1] \to \C$
in
$\Pi_{\Om_{\hat v}}^{\de,L}$
with $L>0$ and $\de \in (0,\rho]$
satisfying the following condition:
\begin{eq-text}\label{4.1}
there exists $a\in(0,1)$ such that 
$\ga|_{[a,1]}$ is $C^1$,
$| \ga(a)| = \rho$ and
$\ga(t) = \ga(a)a^{-1}t$ for $t\in[0,a]$.
\end{eq-text}

\begin{nota}
Given $T=(V,E)\in\sT$,
we set 
%
\begin{align*}
\De_{T,v}^0 &:= \big\{ (s_{u})_{u\in V}
\in \De_T
\mid s_v = 0 \big\},
\\
\De_{T,v}^1 &:= \big\{ (s_{u})_{u\in V}
\in \De_T \mid 
\sum_{u\in V^1_v}s_u
= s_v  \big\},
\\
\cN_v^0 &:= \big\{ (\un\xi{}_{u})_{u\in V}
\in X_\Om^T
\mid \un\xi{}_v = \un0{}_{\Om_v} \big\},
\\
\cN_v^1 &:= \big\{ (\un\xi{}_{u})_{u\in V}
\in X_\Om^T \mid 
\sum_{u\in V^1_v}\fp_{\Om_u}(\un\xi{}_u) 
= \fp_{\Om_v}(\un\xi{}_v)  \big\},
%
%
%
\end{align*}
where $\De_{T,v}^0$ and $\cN_v^0$
(resp. $\De_{T,v}^1$ and $\cN_v^1$)
are defined for $v\in V^\circ$ (resp. $v\in V\setminus L$).
\end{nota}

\begin{dfn}\label{dfn:4.2}
We call a family 
of maps 
$
\Psi_t : \De_T \to X_\Om^T
$
$\big(t\in[a,1]\big)$
\emph{$(\ga,T)$-adapted 
deformation of 
$\sD\big( \ga(a) \big)$ in $X_\Om^T$}
%
if it satisfies the following conditions:
\begin{enumerate}

\item
$\Psi_a = \sD\big( \ga(a) \big)$,

\item
the map
$
\big(t,\vec{s}\,\big) \in [a,1] \times \De_T
\mapsto
\Psi_t\big(\vec{s}\,\big) \in X_\Om^T$
is locally Lipschitz,

\item
the $v$-th component $\un\xi{}^t_v$
of $\vec{\un\xi{}^t}:=\Psi_t\big(\vec{s}\,\big)$
depends only on
the variables 
$
s_u
$
$(u\in W_v)$,
where
$$
W_v := \{\hat v\}\cup
\bigcup_{v'}V_{v'}^1
$$
and the union $\bigcup_{v'}$
is taken over all the vertexes
$v'$ on the path from $v_\upa$
to $\hat v$,

\item
$\un\xi{}_{\hat{v}}^t(\vec{s}\,)=\un\ga(t)$
holds for any $t\in[a,1]$ and $\vec{s}\in\De_T$,

\item
$
\Psi_t\big( \De_{T,v}^0
\big) 
\subset \cN_v^0
$
holds for any $t\in [a,1]$
and $v\in V^\circ$,

\item
$
\Psi_t\big( \De_{T,v}^1
\big) 
\subset \cN_v^1
$
holds for any $t\in [a,1]$
and $v\in V\setminus L$.

\end{enumerate}

\end{dfn}

We now show the folowing
\begin{prp}\label{prp:4.3}
Consider
an iterated convolution $\hat{\psi}_{T}$ of 
$\big(\,T\,;\{\hat{f}_v\}_{v\in V},\{\hat{\varphi}_v\}_{v\in V}\big)$
with the data
$T=(V,E)\in\mathscr{T}$,
$\{\hat{f}_v\}_{v\in V}\subset\hat\sR_{\O}$
and
$\hat{\varphi}_v\in\hat\sR_{\Om_v}$ $(v\in V)$
and assume that a $(\ga,T)$-adapted deformation
$(\Psi_t)_{t\in[a,1]}$ of $\sD\big( \ga(a) \big)$
is given.
Then, 
the analytic continuation of
$
\fp_{\Om_{\hat v}}^*\hat \psi_T
$
along $\un\ga$ is written as follows:
$$
(\fp_{\Om_{\hat v}}^*\hat \psi_T)\big( \uga(t) \big) = 
\big( \Psi_t \big)_\# [\De_T](\be_T)
%
\quad\text{for}\quad
t\in[a,1].
$$
\end{prp}

Notice that
the situation discussed in \cite{S3} and \cite{KS}
is regarded as the case
where the iteration diagram $T\in\sT_k$ satisfies
$V_{\hat v}^1=V^\circ$
with $\hat f_{\hat v}=\hat \ph_{\hat v}=1$.
We first note the following
\begin{lmm}[\cite{S3}]\label{lmm:4.4}
Let $T=(V,E)$ be an iteration diagram
such that $V_{\hat v}^1=V^\circ$
and assume that a $(\ga,T)$-adapted deformation
$(\Psi_t)_{t\in[a,1]}$ of $\sD\big( \ga(a) \big)$ is given.
Then, 
the analytic continuation of
$
\fp_{\Om_{\hat v}}^*\hat \psi_T
$
along $\un\ga$ is written as follows:
$$
(\fp_{\Om_{\hat v}}^*\hat \psi_T)\big( \uga(t) \big) = 
\big( \Psi_t  \big)_\# [\De_T](\be_T)
\quad\text{for}\quad
t\in[a,1].
$$
\end{lmm}
Indeed,
in this case,
we can regard $\beta_T$ as a holomorphic
$(k-1)$-form on $X_\Om^{k-1}$
with a holomorphic parameter
$\un\xi{}_{\hat v}\in X_{\Om^{*(k-1)}}$.
Therefore,
adapting the proof of \cite[Prop.~5.4]{S3}
and restricting the parameter $\un\xi{}_{\hat v}$
to $\un\ga(t)$,
we obtain Lemma \ref{lmm:4.4}.

\begin{nota}
Let a $(\ga,T)$-adapted 
deformation $(\Psi_t)_{t\in[a,1]}$ 
of $\sD\big( \ga(a) \big)$ be given
and assume that
$
(s_u)_{u\in V\setminus V_v^{\circ}}
$
is taken so that it satisfies
\begin{equation}\label{4.2}
\big\{
\big((s_u)_{u\in V \setminus V_v},
(s_v\ti s_u)_{u\in V_v}\big)
\mid
(\ti s_u)_{u\in V_v}
\in \De_{T_v}
\big\}
\subset
\Delta_T.
\end{equation}
Let ${\rm pr}_{T_v}:X_\Om^T\to X_\Om^{T_v}$
be the natural projection.
We define a map
$$
\Psi_t|_{T_v}
\big((s_u)_{u\in V\setminus V_v^{\circ}};
\,\cdot\,\big)
:\De_{T_v}
\to X_\Om^{T_v}
$$
by
$$
\Psi_t|_{T_v}((s_u)_{u\in V\setminus V_v^{\circ}};
\vec{s'}\,)
:=
{\rm pr}_{T_v}\circ\Psi_t\big(
(s_u)_{u\in V\setminus V_v},
(s_vs_u)_{u\in V_v}\big)
$$
for 
$
\vec{s'}:=(s_u)_{u\in V_v}
\in\De_{T_v}.
$
We note that the map $\Psi_t|_{T_v}$ depends only
on the parameters 
$s_u$
$
\big(u\in W_v\big),
$
and hence,
we write the map as
$
\Psi_t|_{T_v}((s_u)_{u\in W_v};
\,\cdot\,\big).
$
(See Definition \ref{dfn:4.2}.3.)
\end{nota}
Since the path 
$
\un\ga{}_v((s_u)_{u\in W_v};
\,\cdot\,)
:[0,1]\to X_{\Om_v}
$
obtained by concatenating
the paths 
$
t\in[0,a]\mapsto \sL_v(\ga(a)s_vt/a)
$
and the $v$-th component of
the path
$
t\in[a,1]\mapsto
\Psi_t|_{T_v}((s_u)_{u\in W_v};
\vec{s'}\,)
$
is independent of the choice of
$
(s'_u)_{u\in V_v^\circ},
$
we find that,
for fixed $(s_u)_{u\in W_v}$,
\begin{eq-text}\label{4.3}
$
\big(\Psi_t|_{T_v}\big)_{t\in[a,1]}
$
defines a $(\ga_v,T_v)$-adapted 
deformation of $\sD\big( \ga_v(a) \big)$.
\end{eq-text}

\begin{proof}[Proof of Proposition \ref{prp:4.3}]
We prove the following for every $v\in V$
by the induction used in the proof of
Proposition \ref{prp:3.7}:
\begin{equation}\label{4.4}
(\fp_{\Om_{v}}^*\hat \psi_{T_v})\big( \un\ga{}_v(t) \big) = 
\big( \Psi_t|_{T_v}\big)_\# [\De_{T_v}](\be_{T_v})
\end{equation}
holds for every $t\in[a,1]$ and 
$(s_u)_{u\in W_v}$
satisfying the condition \eqref{4.2}.
For the case $v\in L$,
the equation \eqref{4.4} follows from the definition
of $\ga_v$.

Now, assume that
the equation \eqref{4.4} holds for
all the vertexes $u\in V^\circ_v$.
Let $\wt T_v$ be an iteration diagram
defined by the vertexes 
$\wt V_v:=V_v^1\cup \{v\}$
and the edges of the form $u\to v$ $(u\in V_v^1)$.
We set 
$
\Psi_t|_{\wt T_v}
:=
{\rm pr}_{\wt T_v}\circ
\Psi_t|_{T_v},
$
where
$
{\rm pr}_{\wt T_v}:X_\Om^{T_v}\to 
X_\Om^{\wt T_v}
:=
\prod_{u\in \wt V_v} X_{\Om_{u}}
$
is the natural projection.
Since 
$\Psi_t|_{\wt T_v}((s_u)_{u\in W_v};
\vec{s'}\,)$ 
depends only on
the variables $s_u$
$
\big(u\in \wt V_v\big),
$
it defines a map
$
\Psi_t|_{\wt T_v}
:\De_{\wt T_v}
\to X_\Om^{\wt T_v}.
$
We then obtain 
$$
\big( \Psi_t|_{T_v}\big)_\# [\De_{T_v}](\be_{T_v})
=
\big( \Psi_t|_{\wt T_v} \big)_\# [\De_{\wt T_v}](\ti\be_{\wt T_v})
$$
from the induction hypothesis,
where $\ti\be_{\wt T_v}$ is
a holomorphic $1$-form defined by
$$
\ti\be_{\wt T_v}:=
(\fp_{\Om_{v}}^*\hat{\varphi}_v)(\un\xi{}_v)
\hat{f}_v\big(\xi_v -\sum_{u\in V_v^1}\xi_u\big)
\prod_{u\in V_v^1}
(\fp_{\Om_{u}}^*\hat{\psi}_{T_u})(\un\xi{}_u)
\bigwedge_{ u\in V_v^1}\dd \un\xi{}_{u}.
$$
Applying Lemma \ref{lmm:4.4},
we have
\begin{align*}
\big( \Psi_t|_{\wt T_v} \big)_\# [\De_{\wt T_v}](\ti\be_{\wt T_v})
&=
\fp_{\Om_{ v}}^*
\Big(
\hat{\varphi}_v\cdot
\Big(\hat{f}_v*
\sideset{}{^*}\prod_{u\in V_v^1}
\hat{\psi}_{T_u}
\Big)\Big)
\big( \un\ga{}_v(t) \big)
\\
&=
(\fp_{\Om_{v}}^*\hat \psi_{T_v})\big( \un\ga{}_v(t) \big).
\end{align*}
It proves \eqref{4.4} for all $v\in V$.
Since $\ga_{\hat v}=\ga$ and $T_{\hat v}=T$,
we obtain Proposition \ref{prp:4.3}.
\end{proof}

The proof of Theorem \ref{thm:3.9}
is reduced to the following
\begin{thm}\label{thm:4.6}
Let $T=(V,E)$ be an iteration diagram in $\sT_k$
$(k\geq 1)$ and assume that 
$\ga\in\Pi_{\Om_{\hat{v}}}^{\de,L}$
for $L>0$ and $\de \in (0,\rho]$
satisfying the condition \eqref{4.1} is given.
Then, there exists a $(\ga,T)$-adapted deformation
$(\Psi_t)_{t\in[a,1]}$ of $\sD\big( \ga(a) \big)$
such that
\begin{equation}\label{4.5}
\Psi_{t} (\De_T)
\subset
\bigcup_{
\vec{s}
\in\Delta_T
}
\prod_{v\in V}
K_{\Omega_v}^{\delta'(t),s_{v}L(\ga|_t)}.
\end{equation}
Further,
the partial derivatives 
$\pa\xi_v^t/\pa s_u$
are defined almost everywhere on $\De_T$
and
satisfy
\begin{equation}\label{4.6}
\Big|
\Det\Big[ \frac{\pa\xi_v^t}{\pa s_u}\Big]_{u,v\in V^\circ}
\Big|
\leq
\big(c(t)\big)^{k-1},
\end{equation}
where
\begin{equation}\label{4.7}
\delta'(t)
:= \rho\, e^{-2\sqrt{2}\delta^{-1}L_a(\ga_{|t})},
\qquad
c(t) :=
\rho \, e^{3\de^{-1} L_a(\ga_{|t})
},
\end{equation}
$$
L_a(\ga_{|t})
=\int_a^t|\gamma'(t')|dt'.
$$
\end{thm}

See \cite{KS} for the reduction of Theorem \ref{thm:3.9}
to Theorem \ref{thm:4.6}.
The proof of Theorem \ref{thm:4.6}
will be given in Section \ref{sec:5}
and \ref{sec:6}.

\section{Construction of a
$(\ga,T)$-adapted deformation}\label{sec:5}

In this section,
we construct 
a $(\ga,T)$-adapted deformation
of $\sD\big( \ga(a) \big)$ 
satisfying the conditions in Theorem \ref{thm:4.6}
by the method developed in
\cite{S2}, \cite{S3}, \cite{OD} and \cite{KS}.
Let us assume that
$T=(V,E)\in\sT_k$, a \dfs\ $\Om$ and
$\ga\in\Pi_{\Om_{\hat v}}^{\de,L}$ 
satisfying the assumptions in Theorem \ref{thm:4.6}
are given.
\begin{nota}
We define functions
$\eta_v$ and $D_v$ $(v\in V)$ by
$$
\zeta = (\la,\xi) \in \Rp\times\C \mapsto 
\eta_v(\zeta) :=
\dist\big( (\la,\xi), \{ (0,0) \} \cup \ov\cS_{\Om_v} \big),
$$
$$
\big(\ze_v,(\zeta_u)_{u\in V_v^1}\big)
\in (\Rp\times\C)^{|V_v^1|+1}
\mapsto 
D_v\big(\ze_v,(\zeta_u)_{u\in V_v^1}\big):=
\sum_{u\in V^1_v}\eta_u(\zeta_u) +
\Big| \zeta_v- \sum_{u\in V_v^1}\zeta_u \Big|,
$$
where $\dist(\cdot,\cdot)$ is the Euclidean distance in 
$\R\times\C\simeq \R^3$.
\end{nota}
By the definition of $\eta_v$,
we find that $\eta_v$ is Lipschitz continuous on
$\Rp\times\C$.
More precisely, we have the following:
\begin{equation*}
|\eta_v(\ze)-\eta_v(\ze')|
\leq|\ze-\ze'|
\quad
\text{holds for every}
\quad
\ze,\ze'
\in\Rp\times\C.
\end{equation*}
We then see that
the following holds for every
$\ze_v$, $\ze'_v$ 
and $\ze_u$, $\ze'_u$ $(u\in V_v^1)$
in $\Rp\times\C$:
\begin{equation*}
|D_v\big(\ze_v,(\zeta_u)_{u\in V_v^1}\big)
-D_v\big(\ze'_v,(\zeta'_u)_{u\in V_v^1}\big)|
\leq
|\ze_v-\ze'_v|+
2\sum_{u\in V^1_v}|\ze_u-\ze'_u|.
\end{equation*}
\begin{lmm}\label{lmm:5.2}
The functions $\eta_v$ and $D_v$ satisfy
the following inequality for every
$\ze_v$ and $\ze_u$ $(u\in V_v^1)$
in $\Rp\times\C$:
\begin{equation}\label{5.1}
D_v\big(\ze_v,(\zeta_u)_{u\in V_v^1}\big)
\geq
\eta_v(\ze_v).
\end{equation}
\end{lmm}
\begin{proof}
For each 
$
\zeta_u \in \Rp\times\C
$
$
(u\in V),
$
we can take
$
\ti\ze_u\in
\ov\cS_{\Om_u}\cup\{(0,0)\}
$
so that
$
\eta_u(\ze_u)=|\ze_u-\ti\ze_u|
$
holds.
Therefore,
we find the following inequality:
$$
D_v\big(\ze_v,(\zeta_u)_{u\in V_v^1}\big)
=
\sum_{u\in V^1_v}
|\ze_u-\ti\ze_u|+
\Big| \zeta_v- \sum_{u\in V_v^1}\zeta_u \Big|
\geq
\Big| \zeta_v- \sum_{u\in V_v^1}\ti\zeta_u \Big|.
$$
Since 
$\dst
\sum_{u\in V_v^1}\ti\zeta_u
\in\ov\cS_{\Om_v}\cup\{(0,0)\},
$
we obtain \eqref{5.1}.
\end{proof}

We now define
a family of functions
$$
\ze^t_v:
\De_T\to
\cM_{\Om_v}
\quad
(t\in[a,1],v\in V)
$$
by the following process:
We first define $\ze^t_{\hat{v}}$ by
$
\ze^t_{\hat{v}}(\vec{s}\,)
:=
\ti\ga(t)
$
for all $\vec{s}\in\De_T$.
Next,
assume that $\ze_v^t=(\la_v^t,\xi_v^t)$ has already determined
and that
\begin{eq-text}\label{5.2}
$\ze_v^t(\vec{s}\,)$
is $C^1$ on $[a,1]$ 
and $\la_v^t(\vec{s}\,)$ is increasing
for each $\vec{s}\in\De_T$,
\end{eq-text}
\begin{eq-text}\label{5.3}
$\la_{v}^a(\vec{s}\,)=0$
if and only if $s_v=0$.
\end{eq-text}
Then,
we define $\ze_u^t$ $(u\in V_v^1)$
as follows:
\begin{enumerate}[--]
\item
When $s_v=0$,
we set
$\ze^t_u(\vec{s}\,):=0$
for $t\in[a,1]$.

\item
When $s_v>0$,
we define $\ze_u^t(\vec{s}\,)$ $(u\in V_v^1)$
by the solution of the differential equation
\begin{equation}\label{5.4}
\frac{d\ze_u}{dt}
= \frac{\eta_u(\ze_u)}
{D_v\big(\ze_v^t,(\zeta_u)_{u\in V_v^1}\big)} 
\frac{\pa\ze^t_{v}}{\pa t}
\tag{5.4.$u$}
\end{equation}
with the initial condition
$
\ze_u|_{t=a}=
s_u\ti\ga(a).
$
\end{enumerate}
Notice that,
by the induction hypothesises
and Lemma \ref{lmm:5.2},
we can take $\eps>0$ so that
$D_v\big(\ze_v^t,(\zeta_u)_{u\in V_v^1}\big)\geq\eps$
holds for every $t\in[a,1]$ and $\ze_u\in\Rp\times\C$
$(u\in V_v^1)$ when $s_v>0$.
We then see that
the right hand side of 
(5.4.$u$)  %
is locally Lipschitz continuous
on the variables $\ze_u$ $(u\in V_v^1)$.
Therefore,
since 
$
s_u\ti\ga(a)
\notin
\{
\ze_u\in\Rp\times\C
\mid
\eta_u(\ze_u)=0
\}
=\ov\cS_{\Om_u}\cup\{(0,0)\}
$
when $s_u\neq0$,
the Cauchy-Lipschitz theorem yields
the existence and uniqueness of
the solutions
$
\ze_u:[a,1]\to\cM_{\Om_u}
\setminus\{(0,0)\}
$
$(u\in V_u^1)$
of \eqref{5.4} satisfying
the initial condition
$
\ze_u|_{t=a}=
s_u\ti\ga(a).
$
We immediately derive from the construction
of $\ze_u^t$ $\big(u\in V_v^1\big)$ that 
they also satisfy the induction hypothesises
\eqref{5.2} and \eqref{5.3}.
In such a way,
we can inductively determine
a family of functions
$\ze^t_v$ $(v\in V)$.

We now define a family of maps
$$
\Phi_t:\De_T\to\cM_{\Om}^T
:=\prod_{v\in V}\cM_{\Om_v}
\quad
(t\in[a,1])
$$
by 
$
\Phi_t:=\big(\ze_v^t\big)_{v\in V}.
$
Since each $\ze_v^t(\vec{s}\,)$ satisfies
$$
\la_v^t(\vec{s}\,)
=s_v|\ga(a)|+
\int_a^t\Big|\frac{\pa\xi_v^{t'}}{\pa t'}(\vec{s}\,)\Big|dt',
$$
we find that the path
$
\ti\ga_v^\cdot(\vec{s}\,):
[0,1]\to\cM_{\Om_v}
$
obtained by concatenating the paths
$s_v\ti\ga|_{[0,a]}$
and 
$
\ze_v^\cdot(\vec{s}\,):[a,1]\to\cM_{\Om_v}
$
defines an $\Om_v$-allowed path $\ga_v^\cdot(\vec{s}\,)$
and its lift $\un\ga{}_v^\cdot(\vec{s}\,)$ on $X_{\Om_v}$.
We then define a family of maps
$
\Psi_t:\De_T\to X_\Om^T
$
$(t\in[a,1])$
by
$
\Psi_t(\vec{s}\,):=
\big(\un\ga{}_v^t(\vec{s}\,)\big)_{v\in V}.
$

We now confirm that $\big(\Psi_t\big)_{t\in[a,1]}$ gives
a $(\ga,T)$-adapted deformation of $\sD\big( \ga(a) \big)$.
It suffices to show the following properties 
of $\big(\Phi_t\big)_{t\in[a,1]}$:
\begin{lmm}\label{lmm:5.3}
For each $v\in V$ and $t\in[a,1]$,
the function
$\ze_v^t$ on $\De_T$
only depends on the variables $s_u$
$\big(u\in W_v\big)$.
\end{lmm}
\begin{lmm}\label{lmm:5.4}
For every $t\in[a,1]$,
the followings hold:
\setcounter{equation}{4}
\begin{equation}\label{5.5}
\ze_v^t=0
\quad\text{when}\quad
s_v=0
\quad(v\in V^\circ),
\end{equation}
\begin{equation}\label{5.6}
\sum_{u\in V_v^1}\ze_u^t
=\ze_v^t
\quad\text{when}\quad
\sum_{u\in V_v^1}s_u=s_v
\quad(v\in V\setminus L).
\end{equation}
\end{lmm}
\begin{prp}\label{prp:5.5}
The map $\Phi_\cdot:[a,1]\times\De_T\to \cM_\Om^T$
is Lipschitz continuous.
\end{prp}
Lemma \ref{lmm:5.3} immediately follows 
from the construction of $\ze_v^t$.
\begin{proof}[Proof of Lemma \ref{lmm:5.4}]
Since $\eta_v(0)=0$ for
$v\in V^\circ$,
we find that the hyperplane defined by
$\ze_v=0$ is invariant by the flow of
(5.4.$v$)  %
when $s_{v_\upa}>0$.
Therefore, combining the case $s_{v_\upa}=0$,
we obtain \eqref{5.5}.
We then prove \eqref{5.6}.
We first consider the case
$s_v>0$.
Since $\ze_u^t$ $(u\in V_v^1)$ satisfies 
(5.4.$u$), %
we find that
$
\Xi:=\ze_v^t-
\sum_{u\in V_v^1}\ze_u^t
$
satisfies
\begin{equation}\label{5.7}
\frac{d\,\Xi}{dt}
:=
\frac{|\,\Xi\,|}
{D_v\big((\zeta_u^t)_{u\in V_v^1},\Xi\,\big)} 
\frac{\pa\ze^t_{v}}{\pa t}
\end{equation}
and an initial condition
$
\Xi|_{t=a}
=\big(s_v-\sum_{u\in V_v^1}s_u\big)\ti\ga(a),
$
where
$$
D_v\big((\zeta_u)_{u\in V_v^1},\Xi\,\big)
:=
\sum_{u\in V^1_v}\eta_u(\zeta_u) +
|\, \Xi \,|.
$$
Therefore,
since $\Xi=0$ also satisfies \eqref{5.7}
and the initial condition when 
$
\sum_{u\in V_v^1}s_u=s_v,
$
we obtain \eqref{5.6} from the uniqueness of
such solutions.
On the other hand,
the equation \eqref{5.6} is obvious when $s_{v}=0$
because $\ze_v^t=0$ and $\ze_u^t=0$ $\big(u\in V_v^1\big)$.
\end{proof}
\begin{rem}
Since the inequality
$$
\sum_{u\in V^1_v}\eta_u(\zeta_u)
\leq
D_v\big(\ze_v,(\zeta_u)_{u\in V_v^1}\big)
$$
holds for
$
\ze_v,\ze_u\in\Rp\times\C,
$
we can derive from
(5.4.$u$)  %
the following inequality:
\begin{equation}\label{5.8}
\sum_{u\in V_v^1}\la_u^t
\leq\la_v^t
\quad
\text{holds for}
\quad
t\in[a,1].
\end{equation}
\end{rem}
For the proof of Proposition \ref{prp:5.5},
we use the following
\begin{lmm}\label{lmm:5.7}
If $\la_v^t$ $(v\in V\setminus L)$
satisfies $0<\la_{v}^t(\vec{s}\,)\leq\rho$
at $\vec{s}\in\De_T$ and $t\in[a,1]$,
$\ze_u^t(\vec{s}\,)$ $(u\in V_v^1)$
is written as
$\ze_u^t(\vec{s}\,)=s_us_{v}^{-1}\ze_{v}^t(\vec{s}\,)$.
\end{lmm}
\begin{proof}
Since $\la_u^\cdot$ $(u\in V_v^1)$ satisfies
$\la_u^{t'}(\vec{s}\,)\leq\la_v^t(\vec{s}\,)\leq\rho$
for $t'\in[a,t]$,
we find that 
$\eta_u\big(\ze_u^{t'}(\vec{s}\,)\big)=|\ze_u^{t'}(\vec{s}\,)|$
holds by the definition of $\eta_u$.
Therefore,
since $\ze_u=s_us_{v}^{-1}\ze_{v}^t$ also gives
a solution of 
(5.4.$u$)  %
satisfying the initial condition 
$
\ze_u|_{t=a}=
s_u\ti\ga(a),
$
we obtain the representation of $\ze_u^t$
from the uniqueness of such solutions.
\end{proof}
\begin{proof}[Proof of Proposition \ref{prp:5.5}]
We prove by induction
that each of the functions
$\ze_v^t$ $(v\in V)$ is Lipschitz continuous
on $\De_T$.
Since $\ze_{\hat v}^t$ is independent of
$\vec{s}\in\De_T$,
the statement holds for $v=\hat{v}$.
We then assume that $\la_v^t$ is
Lipschitz continuous on $\De_T$.
Since the right hand side of
(5.4.$u$)  %
$(u\in V_v^1)$
is Lipschitz continuous on the variables
$s_{v'}$ $(v'\in W_v)$
by the induction hypothesis,
we find by the use of the Cauchy-Lipschitz theorem
that $\ze_u^t$ is locally Lipschitz continuous
on the variables $s_{v'}$ $(v'\in W_u)$
when $s_v>0$.
Further, 
we see by the use of the representation
of $\ze_u^t$ in Lemma \ref{lmm:5.7} that
the following holds for $\vec{s},\vec{s'}\in\De_T$
with $s_v,s'_v$ sufficiently small:
$$
\big|\ze_u^t(\vec{s}\,)-\ze_u^t(\vec{s'}\ )\big|
\leq
\frac{\big|\ze_v^t(\vec{s}\,)\big|}{s_v}
|s_u-s'_u|
+
\frac{s'_u}{s_v}
\big|\ze_v^t(\vec{s}\,)-\ze_v^t(\vec{s'}\,)\big|
+
\frac{s'_u}{s_v}
\frac{\big|\ze_v^t(\vec{s'}\,)\big|}{s'_v}
|s'_v-s_v|.
$$
We may assume without loss of generality
that $s'_v\leq s_v$, and hence,
$s'_u\leq s_v$.
Therefore, we find by the induction hypothesis
that we can take the Lipschitz constant
of $\ze_u^t$ uniformly in a neighborhood of $s_v=0$.
Thus, we have the Lipschitz continuity of $\ze_u^t$
on $\De_T$.
Since $\ze_u^t(\vec{s}\,)$ is $C^1$ on $[a,1]$
for each $\vec{s}\in\De_T$,
we obtain the Lipschitz continuity of
the map $\Phi_\cdot$.
\end{proof}

\section{Estimates for
the $(\ga,T)$-adapted deformation}\label{sec:6}

In the previous section,
we constructed the map
$\Phi_\cdot:[a,1]\times\De_T\to \cM_\Om^T$
for $\ga\in\Pi_{\Om_{\hat v}}^{\de,L}$ 
satisfying the assumptions in Theorem \ref{thm:4.6}
and the properties of the
$(\ga,T)$-adapted deformation
$\big(\Psi_t\big)_{t\in[a,1]}$
was reduced to its properties.
In this section,
we derive estimates for $\Phi_\cdot=(\ze_v^\cdot)_{v\in V}$
which are necessary to prove Theorem \ref{thm:4.6}.
We first show the following
\begin{lmm}\label{lmm:6.1}
For every $v\in V$,
$t\in[a,1]$ and $\vec{s}\in\De_T$,
the function
$\ze_v^t(\vec{s}\,)$ satisfies
\begin{equation}\label{6.1}
\dist\big( \ze_v^t(\vec{s}\,),  \ov\cS_{\Om_v} \big)
\geq
\de'(t),
\end{equation}
where $\de'(t)$ is the function given in \eqref{4.7}.
\end{lmm}
\begin{proof}
When $s_v=0$,
\eqref{6.1} immediately follows from \eqref{5.5}.
We then assume that $s_v>0$.
Let $\{v_j\}_{j=1}^{\ell}$ be vertexes
such that 
$v=v_1\to v_{2}\to \cdots \to v_\ell=\hat{v}$ gives a path
on $T$ from $v$ to $\hat{v}$.
Since $s_{v_j}\geq s_v$ holds for $j=1,\cdots,\ell$,
we find that $\ze^t_j:=\ze^t_{v_j}$ satisfies
$\eta_{j}\big(\ze^t_j(\vec{s}\,)\big)>0$ for every $t\in[a,1]$,
where $\eta_j:=\eta_{v_j}$.
We now consider estimates of
the function
$h_j^t:=\eta_{j}\big(\ze^t_j(\vec{s}\,)\big)$.
Since $h_1^t$ is Lipschitz continuous on $[a,1]$,
we find by Rademacher's theorem that
the derivative of $h_1^t$ exists a.e. on $[a,1]$
and satisfies
$
|dh_1^t/dt|
\leq
|\pa\ze_1^t/\pa t|.
$
We further obtain from
Lemma \ref{lmm:5.2} and
(5.4.$v_j$)
the following sequence of inequalities:
\begin{equation}\label{6.2}
\frac{1}{h_1^t}
\Big|\frac{\pa\ze^t_1}{\pa t}\Big|
\leq
\frac{1}{h_2^t}
\Big|\frac{\pa\ze^t_2}{\pa t}\Big|
\leq\cdots\leq
\frac{1}{h^t_{\ell}}
\Big|\frac{\pa\ze^t_\ell}{\pa t}\Big|
\leq
\frac{|\ti\ga'|}{\de}.
\end{equation}
Thus,
we have the following:
\begin{equation}\label{6.3}
\frac{1}{h_1^t}
\Big|
\frac{dh_1^t}{dt}
\Big|
\leq
\frac{|\ti\ga'|}{\de}
\quad
\text{holds a.e. on $[a,1]$.}
\end{equation}
Since
$h_1^a=\sqrt{2}s_v\rho$
and
$
|\ti\ga'|
=\sqrt{2}|\ga'|,
$
we derive from \eqref{6.3} the following:
\begin{equation*}
\sqrt{2}s_v\rho 
e^{-\sqrt{2}\delta^{-1}L_a(\ga_{|t})}
\leq
h_1^t
\leq
\sqrt{2}s_v\rho 
e^{\sqrt{2}\delta^{-1}L_a(\ga_{|t})}.
\end{equation*}
Therefore, 
we find that
$
\rho 
e^{-2\sqrt{2}\delta^{-1}L_a(\ga_{|t})}
\leq
h_1^t
$
holds for any $t\in[a,1]$
when 
$
\sqrt{2}s_v\geq
e^{-\sqrt{2}\delta^{-1}L_a(\ga_{|t})}.
$
Otherwise,
we have $h_1^t\leq\rho$,
and hence,
$
\big|\ze_1^t\big|\leq\rho
$
holds for any $t\in[a,1]$.
It proves \eqref{6.1}.
\end{proof}
We next show that
$\Phi_t=\big((\la_v^t,\xi_v^t)\big)_{v\in V}$
satisfies the estimate \eqref{4.6}.
We use the following
\begin{lmm}\label{lmm:6.2}
Suppose that $\vec{s},\vec{s'}\in\De_T$
satisfy $s_u=s'_u$ for $u\in V\setminus V_v^1$
with $v\in V\setminus L$.
Then, we have for every $t\in[a,1]$
\begin{equation}\label{6.4}
\sum_{u\in V_v^1}
\Big|
\frac{\pa\xi_u^t}{\pa t}(\vec{s}\,)
-\frac{\pa\xi_u^t}{\pa t}(\vec{s'}\,)
\Big|
\leq
3\frac{|\ga'|}{\de}
\sum_{u\in V_v^1}
\big|\xi_u^t(\vec{s}\,)-\xi_u^t(\vec{s}\,)\big|.
\end{equation}
\end{lmm}
\begin{proof}
Since the left hand side of
(5.4.$u$)
satisfies
\begin{align*}
&
\sum_{u'\in V_v^1}
\left|\left(
\frac{\eta_{u'}(\ze_{u'})}
{D_v\big(\ze_v^t,(\zeta_u)_{u\in V_v^1}\big)}
-
\frac{\eta_{u'}(\ze'_{u'})}
{D_v\big(\ze_v^t,(\zeta'_u)_{u\in V_v^1}\big)} 
\right)
\frac{\pa\xi^t_v}{\pa t}
\right|
\\
&\leq
\frac{
\big|D_v\big(\ze^t_v,(\ze'_u)_{u\in V_v^1}\big)-
D_v\big(\ze^t_v,(\ze_u)_{u\in V_v^1}\big)\big|}
{D_v\big(\ze^t_v,(\ze'_u)_{u\in V_v^1}\big)}
\Big|\frac{\pa\xi^t_v}{\pa t}\Big|
\sum_{u'\in V_v^1}
\frac{\eta_{u'}(\ze_{u'})}{D_v\big(\ze^t_v,(\ze_u)_{u\in V_v^1}\big)}
\\
&\quad +
\Big|\frac{\pa\xi^t_v}{\pa t}\Big|
\sum_{u'\in V_v^1}
\frac{\big|\eta_{u'}(\ze_{u'})-\eta_{u'}(\ze'_{u'})\big|}
{D_v\big(\ze^t_v,(\ze'_u)_{u\in V_v^1}\big)}
\\
&\leq
\frac{3}{D_v\big(\ze^t_v,(\ze'_u)_{u\in V_v^1}\big)}
\Big|\frac{\pa\xi^t_v}{\pa t}\Big|
\sum_{u'\in V_v^1}
\big|\ze_{u'}-\ze'_{u'}\big|,
\end{align*}
we obtain \eqref{6.4} by the use of \eqref{6.2}
when $s_v>0$.
The inequality \eqref{6.4} is trivial when $s_v=0$.
\end{proof}
We now set for $\vec{s},\vec{s'}\in \De_T$
and $v\in V\setminus L$
$$
\Xi_v(t):=
\sum_{u\in V_v^1}
\big|\xi_{u}^t(\vec{s}\,)-\xi_{u}^t(\vec{s'}\,)\big|.
$$
Under the assumption in Lemma \ref{lmm:6.2},
we obtain from \eqref{6.4} the following:
$$
\big|\Xi_v(t)-\Xi_v(a)\big|
\leq\frac{3}{\de}
\int_a^t\Xi_v(t')|\ga'(t')|dt'.
$$
Therefore,
Gronwall's Lemma yields
the following estimate:
\begin{equation}\label{6.5}
\Xi_v(t)\leq
c(t)\,
\Xi_v(a)
\quad
\text{holds for every}
\quad
t\in[a,1],
\end{equation}
where $c(t)$ is the constant given in \eqref{4.7}.
Since 
$
\Xi_v(a)=
\rho
\sum_{u\in V_v^1}|s_u-s'_u|,
$
we see by \eqref{6.5}
that the following holds a.e. on $\De_T$:
\begin{equation}\label{6.6}
\Big|
\Det\Big[ \frac{\pa\xi_{u'}^t}{\pa s_u}\Big]_{u,u'\in V_v^1}
\Big|
\leq
\prod_{u\in V_v^1}
\Big(
\sum_{u'\in V_v^1}
\Big|\frac{\pa\xi_{u'}^t}{\pa s_u}\Big|\,
\Big)
\leq
\big(c(t)\big)^{|V_v^1|}.
\end{equation}
Notice that $\pa\xi_v/\pa s_u=0$
for $u\notin W_v$,
and hence,
we find the following:
\begin{equation}\label{6.7}
\Big|
\Det\Big[ 
\frac{\pa\xi_{u'}^t}{\pa s_u}
\Big]_{u,u'\in V^\circ}
\Big|
=
\prod_{v\in V\setminus L}
\Big|
\Det\Big[ \frac{\pa\xi_{u'}^t}{\pa s_u}\Big]_{u,u'\in V_v^1}
\Big|.
\end{equation}
Since
$
\sum_{v\in V\setminus L}
|V_v^1|=k-1,
$
we obtain \eqref{4.6} from \eqref{6.6} and \eqref{6.7}.
Together with \eqref{5.8} and Lemma \ref{lmm:6.1},
Theorem \ref{thm:4.6} is validated.

\section{Applications to nonlinear 
differential and difference equations}\label{sec:7}

In this section,
we discuss the resurgence of formal series solutions
$$
\Phi
=
\left(
\begin{array}{c}
\Phi^{(1)}
\\
\vdots
\\
\Phi^{(n)}
\end{array}
\right)
\in \C^n[[x^{-1}]]
$$
of a nonlinear differential equation
\begin{equation}\label{7.1}
\frac{d}{dx}\Phi
=F(x^{-1},\Phi)
\end{equation}
at $x=\infty$
with
$F(x^{-1},\Phi)\in\C^n\{ x^{-1},\Phi\}$
satisfying the conditions
\begin{equation}\label{7.2}
F(0,0)=0
\quad
\text{and}
\quad
{\rm det} \big(\partial_{\Phi}F(0,0)\big)
\neq0.
\end{equation}
Under the assumption \eqref{7.2},
there exists a unique formal series solution
of the form
$$
\Phi(x)
=\sum_{k=1}^{\infty}\Phi_kx^{-k}.
$$
We rewrite $F(x^{-1},\Phi)$ in the following form:
$$
F(x^{-1},\Phi)=
F_0(x^{-1})+
\partial_{\Phi}F(0,0)\Phi
+
\sum_{|\ell|\geq1}
F_\ell(x^{-1})\Phi^{\ell},
$$
where 
$
\Phi^{\ell}:=\big(\Phi^{(1)}\big)^{\ell_1}
\cdots\big(\Phi^{(n)}\big)^{\ell_n}
$
with
$\ell=(\ell_1,\cdots,\ell_n)\in\Z_{\geq0}^n$ and
$|\ell|:=\ell_1+\cdots+\ell_n$.
Regarding \eqref{7.1} as an equation for 
$\tilde{\Phi}(x):=x^{-1}(\Phi(x)-\Phi_1x^{-1})$,
we may assume without loss of generality
that
\begin{equation}
F_{\ell}(0)=0
\quad
\text{for every}
\quad
\ell\in\Z_{\geq0}^n.
\end{equation}
Applying the Borel transform,
\eqref{7.1} is rewritten as follows:
\begin{equation}\label{7.4}
P(\xi)\hat\Phi
=
\hat{F}_0+
\sum_{|\ell|\geq1}
\hat{F}_\ell*\hat\Phi^{*\ell},
\end{equation}
where
$
P(\xi):=-\xi-\partial_{\Phi}F(0,0)
$
and
$
\hat\Phi^{*\ell}
:=\cB\big(\Phi^\ell\big).
$
We now inductively determine 
$\hat{\Phi}_k$ $(k\geq1)$
by the following procedure:
\begin{equation}\label{7.5}
\hat{\Phi}_1:=
P^{-1}\hat{F}_0,
\end{equation}
\begin{equation}\label{7.6}
\hat{\Phi}_{k+1}:=
P^{-1}\sum_{j=1}^k
\sum_{|\ell|=j}
\hat{F}_{\ell}*
\sum_{\substack{
k_1+\cdots+k_j=k
\\
k_i\geq1
}}
\hat{\Phi}_{k_1,\cdots,k_j}^{\ell},
\end{equation}
where
$
\hat{\Phi}_{k_1,\cdots,k_j}^{\ell}
$
is the convolution product of
functions in
$$
\Big\{
\hat{\Phi}_{k_i}^{(m)}
\ \Big|\ 
1+\sum_{p=1}^{m-1}\ell_p
\leq
i
\leq
\sum_{p=1}^{m}\ell_p,
1\leq m\leq n
\Big\}.
$$
We now introduce the following
\begin{dfn}
Let $T=(V,E)\in \sT$ and consider a function
$\nu=(\nu_1,\nu_2):V\to\{1,\cdots,n\}^2$.
We call such a pair $\ov{T}=(T,\nu)$
\emph{$n$-decorated iteration diagram}
and denote the set of $n$-decorated iteration diagrams
by $\sT^n$.
\end{dfn}
Let $\ov{T}=(T,\nu)\in\sT^n$.
We define an equivalence relation $\sim_v$ on $V_v^1$
for $v\in V\setminus L$ as follows:
$u\sim_v u'$ $(u,u'\in V_v^1)$ if $T_u=T_{u'}$ and
$\nu|_{V_u}=\nu|_{V_{u'}}$.
For each $[u]\in V_v^1/\sim_v$,
we define an integer $\# [u]$
as the cardinal of
$
\{u'\in V_v^1\mid u'\sim_v u\}
$
and the multiplicity $\mu_v$
of $\ov{T}$ at $v$ by
$$
\mu_v:=
\Big(\prod_{[u]\in V_v^1/\sim_v}\big(\# [u]\big)!\Big)^{-1}
\cdot
\big(\big|V_v^1\big|\big)!.
$$
We set 
$\mu_{\ov{T}}:=\prod_{v\in V\setminus L}\mu_v$.
We further define a map
$\la_{\ov{T},j}:V\to\Z_{\geq0}$
$(j=1,\cdots,n)$ by 
$
\la_{\ov{T},j}(v):=
|\{u\in V_v^1\mid
\nu_1(u)=j\}|
$
when $v\in V\setminus L$
and 
$
\la_{\ov{T},j}(v):=0
$
when $v\in L$.
We set
$\la_{\ov{T}}=(\la_{\ov{T},1},\cdots,\la_{\ov{T},n})$.
By the use of the
multinomial theorem,
we obtain the following
\begin{lmm}\label{lmm:7.2}
For every $k\geq1$ and $j\in\{1,\cdots,n\}$,
$\hat\Phi_k^{(j)}$ is written as follows:
\begin{equation}\label{7.7}
\hat\Phi_k^{(j)}
=\sum_{\ov{T}\in\sT_k^{n,(j)}}
\mu_{\ov{T}}\cdot
\psi_{\ov{T}},
\end{equation}
where
$
\sT_k^{n,(j)}
:=
\{(T,\nu)\in\ov\sT\mid
T\in\sT_k,\nu_1(\hat{v})=j
\}
$
and
$\psi_{\ov{T}}$
is the iterated convolution product of
$
\big(\,T\,;\{\hat{f}_v\}_{v\in V},\{\hat{\varphi}_v\}_{v\in V}\big)
$
with
$
\hat{f}_v:=
\hat{F}_{\la_{\ov{T}}(v)}^{(\nu_2(v))}
$
and
$
\hat{\ph}_v:=(P^{-1})_{\nu(v)}
$
$\big((\nu_1(v),\nu_2(v))$-th entry
of $P^{-1}\big)$.
%
%
%
%
\end{lmm}
We now show the following
\begin{lmm}\label{lmm:7.3}
There exist positive constants $C$
and $\de(<1/n)$
such that
\begin{equation}\label{7.8}
\sum_{\ov{T}\in\sT_k^{n,(j)}}
\mu_{\ov{T}}
\leq
\de B(k)C^{k}
\end{equation}
holds for ever $k\geq1$ and $j\in\{1,\cdots,n\}$,
where $B(k)$ is a constant defined by
$$
B(k):=
\frac{3}{2\pi^2(k+1)^2}.
$$
\end{lmm}
\begin{proof}
Since the left hand side of \eqref{7.8}
is independent of $j$,
we denote it by $N_k$.
We find by \eqref{7.6} that
$\{N_k\}_{k\geq1}$ satisfy the following:
\begin{equation}\label{7.9}
N_{k+1}=
\sum_{j=1}^k
\sum_{|\ell|=j}
n
\sum_{\substack{
k_1+\cdots+k_j=k
\\
k_i\geq1
}}
N_{k_1}\cdots N_{k_j}.
\end{equation}
Taking $C>0$ sufficiently large,
we may assume that $N_1$ satisfies
\eqref{7.8} for arbitrary small positive $\de$.
We then assume that \eqref{7.8} holds
for $k\leq K$.
Using the inequality
$$
\sum_{
k_1+\cdots+k_j=k
}
B(k_1)\cdots B(k_j)
\leq
B(k),
$$
we derive from \eqref{7.9} the following:
$$
N_{K+1}
\leq
nB(k)C^k
\sum_{j=1}^k
\sum_{|\ell|=j}
\de^j
\leq
nB(k)C^k
\sum_{j=1}^k
(n\de)^j
\leq
nB(k)C^k
\frac{n\de}{1-n\de}.
$$
Therefore,
taking $C$ sufficiently large
so that it satisfies
$n^2(1-n\de)^{-1}\leq C$,
we see that \eqref{7.8} holds
for $k=K+1$.
We thus obtain \eqref{7.8} for $k\geq1$.
\end{proof}
Let $\Om= (\Om_L)_{L\in\Rp}$ be a \dfs\ 
defined by the formula
\begin{equation}\label{7.10}
\Omega_L
=\{ \xi\in\C\ |\ 
{\rm det} \big(-\xi-\partial_{\Phi}F(0,0)\big)=0,
|\xi|\leq L
\}.
\end{equation}
We then find that each entry of $P^{-1}$
is $\Om$-continuable,
and hence, we obtain from
Corollary \ref{crl:3.10}
the following estimates:
for any $\de,L>0$, 
there exist $c,\de'>0$ such that
\begin{equation}\label{7.11}
\big\|
\hat\psi_{\ov{T}}
\big\|_{\Om_{\hat{v}}}^{\de,L}
\leq
\frac{c ^{k-1}}{(k-1)!}
\prod_{v\in V}
\big\|(P^{-1})_{\nu(v)}
\big\|_{\Om}^{\de',L}
\big\|
\hat{F}_{\la_{\ov{T}}(v)}^{(\nu_2(v))}
\big\|_{\O}^{\de',L}
\end{equation}
holds for every $\ov{T}\in\sT_k^{n,(j)}$.
Since $F(x^{-1},\Phi)\in\C^n\{ x^{-1},\Phi\}$,
we can take $A>0$ so that
$
\big\|(P^{-1})_{ij}
\big\|_{\Om}^{\de',L}
\leq A
$
and
$
\big\|\hat{F}_{\ell}^{(j)}
\big\|_{\O}^{\de',L}
\leq A^{1+|\ell|}
$
hold for any $i,j\in\{1,\cdots,n\}$
and $\ell\in\Z_{\geq0}^n$.
Notice that
$
\big|\la_{\ov{T}}(v)\big|
=\big|V_v^1\big|
$
when $v\in V\setminus L$,
and hence,
$
\sum_{v\in V}
\big|\la_{\ov{T}}(v)\big|
=k-1.
$
Therefore,
we derive from
\eqref{7.7}, \eqref{7.8}
and \eqref{7.11}
the following estimates:
there exists a positive constant $C$ such that
\begin{equation}\label{7.12}
\big\|
\hat\Phi_k^{(j)}
\big\|_{\Om^{*\infty}}^{\de,L}
\leq
\frac{C^{k}}{(k-1)!}
\end{equation}
holds for
every $k\geq1$ and $j\in\{1,\cdots,n\}$.
We then find that each entry of
$
\hat{\Phi}=
\sum_{k\geq1}\hat{\Phi}_k
$
converges in
$\hat\sR_{\Om^{*\infty}}$ 
and $\hat{\Phi}$
gives a solution of \eqref{7.4}.
Thus, we obtain the following
\begin{thm}\label{thm:7.4}
Let $\Omega=\{\Omega_L\}_{L\in\Rp}$ 
be a \dfs\ defined by \eqref{7.10}.
Then,
each entry of
the formal series solution
$
\Phi
\in \C^n[[x^{-1}]]
$
of \eqref{7.1} is $\Omega^{*\infty}$-resurgent.
\end{thm}
By the same discussion,
we have the following
\begin{thm}\label{thm:7.5}
Let us consider a nonlinear difference equation
\begin{equation}\label{7.13}
\Phi(x+1)-\Phi(x)
=F(x^{-1},\Phi(x))
\end{equation}
at $x=\infty$ under the assumption \eqref{7.2}.
Then,
there exists a unique formal series solution
$
\Phi(x)
\in \C^n[[x^{-1}]]
$
of \eqref{7.13}
and
each entry of $\Phi(x)$ is 
$\Omega^{*\infty}$-resurgent,
where $\Omega=\{\Omega_L\}_{L\in\Rp}$ 
is a \dfs\ defined by 
\begin{equation}\label{7.14}
\Omega_L
=\{ \xi\in\C\ |\ 
{\rm det} \big((e^{-\xi}-1)-\partial_{\Phi}F(0,0)\big)=0,
|\xi|\leq L
\}.
\end{equation}

\end{thm}




\vspace{.4cm}

\noindent {\em Acknowledgements.}
{
The author expresses his gratitude to Prof. David Sauzin.
His comments and suggestions have been helpful.
The author also expresses his gratitude to
Prof. Masafumi Yoshino and Prof. Yoshitsugu Takei
for their encouragement.
The author is grateful to all the staffs in
Mathematical department of
Hiroshima University
for their kind hospitality.



\end{document}